\documentclass[10pt]{article}
\usepackage{amsthm}
\usepackage[colorlinks=true,linkcolor=blue]{hyperref}
\usepackage{enumitem}
\usepackage{mathrsfs}
\usepackage{bbm}
\usepackage[margin=1.6in]{geometry}
\usepackage{marvosym}
\usepackage{graphicx}

\usepackage{tikz}
\usepackage{url}
\usepackage{amsmath}
\usepackage{amsxtra}
\usepackage{amssymb}
\usepackage{lineno}
\usepackage{turnstile}
\usepackage{verbatim}
\usepackage{enumitem}
\usepackage[retainorgcmds]{IEEEtrantools}
\usepackage{accents}
\usepackage[all]{xy}

\newtheoremstyle{slanted}
  {3pt}      
  {3pt}      
  {\slshape} 
  {}         
  {\bfseries}
  {.}        
  { }        
  {}         

\theoremstyle{definition}

\newtheorem{dfn}{Definition}[section]

\newtheorem{tm}[dfn]{Theorem}

\theoremstyle{slanted}

\theoremstyle{definition}

\newtheoremstyle{claimstyle}
  {3pt}      
  {3pt}      
  {\slshape} 
  {}         
  {\bfseries}
  {.}        
  { }        
  {}         

\theoremstyle{claimstyle}
\newtheorem{clm}{Claim}
\newtheorem*{clm*}{Claim}

\newcommand{\cHull}{\mathrm{cHull}}
\newcommand{\sub}{\subseteq}
\newcommand{\om}{\omega}
\newcommand{\pow}{\mathcal{P}}
\newcommand{\OR}{\mathrm{OR}}
\newcommand{\Hull}{\mathrm{Hull}}

\renewcommand{\diamond}{\diamondsuit}
\newcommand{\rg}{\mathrm{rg}}

\newcommand{\ins}{\trianglelefteq}
\newcommand{\pins}{\triangleleft}
\newcommand{\crit}{\mathrm{cr}}
\newcommand{\rest}{\!\upharpoonright\!}
\newcommand{\com}{\circ}

\newcommand{\Ult}{\mathrm{Ult}}
\newcommand{\sats}{\models}
\newcommand{\J}{\mathcal{J}}
\newcommand{\ZFC}{\mathrm{ZFC}}
\newcommand{\es}{\mathbb{E}}
\DeclareMathOperator{\card}{card}
\newcommand{\myurl}{https://sites.google.com/site/schlutzenberg/home-1}
\newcommand{\HC}{\mathrm{HC}}
\begin{document}
\title{Diamonds in mice}
\author{Farmer Schlutzenberg\footnote{afirstname dot alastname at tuwien dot ac dot at, afirstname dot alastname at gmail dot com,
\url{\myurl}}\\
TU Vienna
}
\maketitle
\begin{abstract}
Let $M$ be a $(0,\om_1+1)$-iterable mouse with no largest cardinal.

Let $\gamma\leq\kappa$ 
be uncountable cardinals of $M$, with $\kappa$ regular in $M$. Then $M\sats\diamondsuit_{\kappa\gamma}^+$, and if $\kappa$ is non-ineffable in $M$ then $M\sats\diamondsuit_{\kappa\kappa}^+$.

Suppose that either $\om_1^M=\om_1$ or $M\sats$ ``I am $(0,\om_1+1)$-iterable''. Let $\kappa\geq\om_1^M$ be a cardinal of $M$. Then $M\sats\diamondsuit^*_{\kappa\om_1^M}$.
\end{abstract}

Jensen introduced the combinatorial principle $\diamond_\kappa$, for  regular uncountable cardinals $\kappa$. This asserts that there is a sequence $\left<A_\alpha\right>_{\alpha<\kappa}$
such that $A_\alpha\sub\alpha$ for each $\alpha<\kappa$ and for every $A\sub\kappa$
there is a stationary set of ordinals $\alpha<\kappa$ such that $A\cap\alpha=A_\alpha$.

A classical result is that if $V=L$ and $\kappa>\omega$ is a regular cardinal then $\diamondsuit_\kappa$ holds.
Kunen showed in \cite[Theorem 11, Chapter 2]{jensen_diamond} that if $\kappa$ is a subtle cardinal then $\diamondsuit_\kappa$ holds.

 Recall that for a regular uncountable cardinal $\kappa$ and uncountable cardinal $\gamma\leq\kappa$,
 $\diamondsuit_{\kappa\gamma}^+$
 asserts that there is a function $F:\pow_\gamma(\kappa)\to V$
 such that for all $x\in\pow_\gamma(\kappa)$,
 $F(x)\sub\pow(x)$,
 $F(x)$ has cardinality $\leq\card(x)$, and for all $A\sub\kappa$
 there is a set $B\sub\kappa$,
 unbounded in $\kappa$,
 such that for all $x\in\pow_\gamma(\kappa)$ of limit ordertype, if $\sup(x\cap B)=\sup(x)$
 then $A\cap x\in F(x)$ and $B\cap x\in F(x)$.

 Recall that a regular cardinal $\kappa$
 is \emph{ineffable} if for all sequences $\left<A_\alpha\right>_{\alpha<\kappa}$
 with $A_\alpha\subseteq\alpha$ for each $\alpha$, there is a stationary set $S\sub\kappa$ such that for all $\alpha,\beta\in S$ with $\alpha<\beta$,
 we have $A_\alpha=A_\beta\cap\alpha$.

Jensen showed in \cite[Theorems 10 and 2, Chapter 2]{jensen_diamond}
that assuming $V=L$ and $\kappa>\om$ is a regular cardinal:
\begin{enumerate}[label=--]
  \item If $\kappa$ is non-ineffable then $\diamondsuit_{\kappa\kappa}^+$ holds.
\item If $\gamma$ is any cardinal with $\aleph_1\leq\gamma<\kappa$
then $\diamondsuit^+_{\kappa\gamma}$
 holds.
 \end{enumerate}

 Schindler and Steel introduced the variant notion $\diamondsuit_{\kappa\gamma}^{+,\mathrm{unctbl}}$,
 which is defined just like $\diamondsuit_{\kappa\gamma}^+$,
 except that one restricts the sets $x\in\pow_\gamma(\kappa)$ considered
 to those that are uncountable.
 They showed in \cite{sile}
 that if $M\sats\ZFC$ is a tame mouse,
 $\kappa>\om$ is a regular cardinal in $M$, and $\gamma$ is an $M$-cardinal with
 $\om_1^M<\gamma<\kappa$, then $\diamondsuit_{\kappa\gamma}^{+,\mathrm{unctbl}}$ holds in $M$.

 \begin{tm}\label{tm:diamond^+}
 Let $M$ be a short extender mouse, $\om_1^M\leq\gamma<\kappa<\kappa^{+M} <\kappa^{+M}+\om<\OR^M$
with $\gamma,\kappa$ being $M$-cardinals and $\kappa$ regular in $M$. Then $M\sats\diamondsuit^+_{\kappa\gamma}$.
 \end{tm}

Note that in comparison with the result of Schindler and Steel in \cite{sile},
we have dropped the tameness assumption on $M$, allow $\gamma=\om_1^M$,
and deal with $\diamondsuit_{\kappa\gamma}^+$, instead of just $\diamondsuit^{+,\mathrm{unctbl}}_{\kappa\gamma}$.

The proof of Theorem \ref{tm:diamond^+} will be very close
to Jensen's proof for $L$, but will need one extra idea at the end of the argument. Likewise for the following version:

\begin{tm}\label{tm:diamond^+_kappa,kappa}
Let $M$ be a short extender mouse, $\om<\kappa<\kappa^{+M}<\kappa^{+M}+\om<\OR^M$,
with $\kappa$ a regular cardinal in $M$,
but $\kappa$ non-ineffable in $M$.
Then $M\sats\diamondsuit_{\kappa\kappa}^+$.
\end{tm}

Let $\om_1\leq\gamma\leq\kappa$ be cardinals (they need not be regular). Recall that $\diamondsuit^*_{\kappa\gamma}$ asserts that there is a function $F:\pow_\gamma(\kappa)\to V$ such that $F(x)$ has cardinality $\leq\card(x)$ for all $x\in\pow_\gamma(\kappa)$, and for every set $A\sub\kappa$ there is a club $C\sub\pow_\gamma(\kappa)$ such that for all $x\in C$, we have $A\cap x\in F(x)$.

Schindler and Steel showed in \cite{sile}
that if $M$ is a tame mouse modelling ZFC, then $\diamondsuit^*_{\kappa\gamma}$ holds in $M$
for all $M$-cardinals $\kappa,\gamma$ with $\omega_1^M<\gamma<\kappa$. They mentioned that
they did not know whether the same held with $\omega_1^M=\gamma$.

\begin{tm}\label{tm:diamond^*}
Let $M$ be a $(0,\om_1+1)$-iterable premouse, and $\om_1^M\leq\kappa<\kappa^{+M}<\OR^M$, where $\kappa$ is an $M$-cardinal.
Suppose that either:
\begin{enumerate}[label=(\roman*)]
\item $\om_1^M=\om_1$, or
\item\label{item:self-it} $M\sats$ ``For every $N\pins M|\kappa^{+M}$ and every $\bar{N}$ of cardinality $\leq\aleph_1$ which can be elementarily embedded into $N$, $\bar{N}$ is $(0,\om_1+1)$-iterable''.
\end{enumerate}
Then $M\sats\diamondsuit^*_{\kappa\om_1^M}$.
\end{tm}

Note that in this result,  $\kappa$ is not assumed to be regular in $M$. In the case that $\kappa$ is regular in $M$, the fact that $M\sats\diamondsuit^*_{\kappa\om_1^M}$
follows easily from the fact that $M\sats\diamondsuit^+_{\kappa\om_1^M}$,
which holds by Theorem \ref{tm:diamond^+}. So it is only in the
case that $\kappa$ is singular in $M$ that we get something new here.

Note that in  the Schindler-Steel result on $\diamondsuit^*_{\kappa\gamma}$, $M$ is assumed tame, and $\gamma>\om_1^M$. 
However, our self-iterability assumption  is a significant one,
and goes beyond what can be proven in general even for tame mice.
(It is shown in \cite{sile}
that all tame mice modelling ``$\om_1$ exists'' model ``there is $\alpha<\om_1$ such that $M|\om_1$ is above-$\alpha$, $(\om,\om_1)$-iterable''.
But in \cite{odle},
there is an example of a proper class tame mouse modelling ``there is no $\alpha<\om_1$ such that $M|\om_1$ is above-$\alpha$, $(\om,\om_1+1)$-iterable''.)
However, this degree of self-iterability does hold for many ``$\varphi$-minimal'' mice. Clearly it implies that $M\sats$ ``$M|\om_1^M$ is $(0,\om_1+1)$-iterable''.
\section{$\diamondsuit^+_{\kappa\gamma}$ and $\diamondsuit^+_{\kappa\kappa}$}
The proofs of Theorems \ref{tm:diamond^+} and \ref{tm:diamond^+_kappa,kappa}
will follow very closely Jensen's original proofs of these facts in
\cite{jensen_diamond}, under the added assumption that $V=L$. (This argument was also used in the Schindler-Steel proof of $\diamondsuit^{+,\mathrm{unctbl}}_{\kappa\gamma}$,
though significant extra argument was
also used there.)
The main point of divergence arises in
one short paragraph at the very end of the proof, and there we need some further argument. That extra argument needs the details of Jensen's construction, so we will give that in detail, for self-containment. There are also some minor extra details pertaining to condensation, but these are routine. Let us also note that in \cite{sile}, it was important that $M$ was sufficiently self-iterable,
and for this, the tameness of $M$
was key, but our proof will avoid any appeal to self-iterability.

\begin{proof}[Proof of Theorem \ref{tm:diamond^+}] We follow the proof of \cite[Theorem 2]{jensen_diamond}, making some very minor adaptations, until the last paragraph of the proof, which will require an extra  (but short) argument.

Working in $M$, we define $F:\pow_{\gamma}(\kappa)\to V$ as follows. Let $X\in\pow_{\gamma}(\kappa)$.
If $X$ does not have limit ordertype then $F(X)=\emptyset$.
Suppose $X$ has limit ordertype.
Let $M_X=\Hull^{M|\kappa}(X\cup\{X,\gamma\})$.
Then we define $F(X)=\{Y\cap X\bigm|Y\in M_X\}$.

We claim this works.
Suppose not. Let $A$ be the $M$-least $A\sub\kappa$ which is a counterexample.
We will define a set $B=\{\beta_\eta\}_{\eta<\kappa}$, cofinal in $\kappa$,
and show that $B$ works for $A$; that is,
for all $X\sub\kappa$ such that $\card(X)<\gamma$, $X$ has limit ordertype and $\sup(X)=\sup(X\cap B)$, we have $A\cap X\in F(X)$ and $B\cap X\in F(X)$. This will be a contradiction.

We first define $\left<N_\nu\right>_{\nu<\kappa}$,
as follows. Let  $N_0$ be the least $N\preccurlyeq M|\kappa^{+M}$ such that $N\cap\kappa$ is transitive and $\gamma\in N$.
Given $N_\nu$ where $\nu<\kappa$,
let $N_{\nu+1}$ be the least $N\preccurlyeq M|\kappa^{+M}$ such that $N\cap\kappa$ is transitive
and $N_\nu\cup\{N_\nu\}\sub N$.
Given $N_\nu$ for all $\nu<\eta$,
where $\eta$ is a limit, let $N_\eta=\bigcup_{\nu<\eta}N_\alpha$.
Now let $\kappa_\nu=N_\nu\cap\kappa$,
$\bar{N}_\nu$ be the transitive collapse of $N_\nu$,
$\pi_\nu:\bar{N}_\nu\to N_\nu$ the uncollapse map, and $\beta_\nu=\OR^{\bar{N}_\nu}$.
Note that $\crit(\pi_\nu)=\kappa_\nu$,
$\pi_\nu(\kappa_\nu)=\kappa$,
and $\kappa_\nu$ is the largest cardinal of $\bar{N}_\nu$,
so by condensation, either:
\begin{enumerate}[label=--]
    \item $M|\kappa_\nu$ is passive and $\bar{N}_\nu=M||\beta_\nu$, or
    \item $M|\kappa_\nu$ is active with extender $F$ (which is necessarily $M$-total) and $\bar{N}_\nu=\Ult(M,F)||\beta_\nu$.
\end{enumerate}

Set $B=\{\beta_\nu\}_{\nu<\kappa}$. We want to see  $B$ works.
So let $X\in\pow_{\gamma}(\kappa)$
have limit ordertype and be such that $\sup(X\cap B)=\sup(X)$.
Note then that $\sup(X)=\kappa_\eta$
for some limit ordinal $\eta$.

Now $F(X)\not\sub \bar{N}_\eta$,
since $\kappa_\eta$ is regular in $\bar{N}_\eta$,
but not regular in $M_X$,
since $X\in M_X$,
and $X$ has ordertype $<\gamma$,
and $\gamma<\kappa_0<\kappa_\eta$
since $\gamma\in N_0$.

Let $A_\eta=\pi_\eta^{-1}(A)$. (Recall $A$ was the $M$-least counterexample selected above, which is definable over $M|\kappa^{+M}$ from the parameter $\gamma$, and so $A\in N_0$.)
Then clearly $A_\eta=A\cap\kappa_\eta$.
Let $B_\eta=B\cap\kappa_\eta$. (We do not claim that $B_\eta\in \bar{N}_\eta$.)
We need to see that $A_\eta,B_\eta\in M_X$. For this it is enough to see that  $\bar{N}_\eta\in M_X$, since
 $A_\eta,B_\eta$ are definable over $\J(\bar{N}_\eta)$ from the parameter $\gamma$.
 And for this, it is enough to see that $\beta_\eta \in M_X$, since either $M|\kappa_\eta$ is passive and $\bar{N}_\eta=M||\beta_\eta\in M_X$,
 or $M|\kappa_\eta$ is active with extender $F$, and hence $F\in M_X$
 (since $\kappa_\eta=\sup(X)$ and $X\in M_X$),
 and $\bar{N}_\eta=\Ult(M,F)||\beta_\eta\in M_X$.

Quite similarly, for each $\nu<\eta$,
$\left<\bar{N}_\xi\right>_{\xi\leq\nu}$ is definable from the parameter $(\beta_\nu,\gamma)$; here if $M|\kappa_\nu$ is active, then  note that 
$\kappa_\nu$ is easily recovered from $\beta_\nu$, since there is no extender in $\es^M$ with index in $(\kappa_\nu,\beta_\nu]$, since $M|\kappa_\nu$ codes a surjection of $\kappa_\nu$ onto $\beta_\nu$). Likewise, 
the natural system of embeddings
$\left<\sigma_{\xi\zeta}\right>_{\xi\leq\zeta\leq\nu}$
with $\sigma_{\xi\zeta}:\bar{N}_\xi\to\bar{N}_\zeta$
(that is, $\sigma_{\xi\zeta}=\pi_\xi\com\pi_\zeta^{-1}$)
is definable from the same parameter.

Let $\widetilde{M}_X$ be the transitive collapse of $M_X$,
and $\tau:\widetilde{M}_X\to M_X$ the uncollapse map.
Since $B\cap X$ is cofinal in $\kappa_\eta$,
we have $\beta_\nu\in\rg(\tau)$ for cofinally many $\nu<\eta$,
and $\gamma\in\rg(\tau)$,
so we have preimages $\widetilde{\bar{N}}_\xi$ under $\tau$ of $\bar{N}_\xi$,
for cofinally many $\xi<\eta$,
and preimages $\widetilde{\sigma}_{\xi\zeta}$ under $\tau$ of $\sigma_{\xi\zeta}$,
for the corresponding $\xi\leq\zeta<\eta$.
Let $\bar{N}^*_\infty$ be the direct limit
of these $\widetilde{\bar{N}}_\xi$'s under these maps $\widetilde{\sigma}_{\xi\zeta}$. Let $\sigma_{\xi\infty}:\widetilde{\bar{N}}_\xi\to\bar{N}^*_\infty$ be the direct limit map. Let $\tau^*:\bar{N}^*_\infty\to \bar{N}_\eta$ be the natural map; that is, $\tau^*\com\widetilde{\sigma}_{\xi\infty}=\sigma_{\xi\eta}\com\tau\rest\widetilde{\bar{N}}_\xi$,
for the appropriate ordinals $\xi<\eta$.
Let $\tau(\widetilde{\kappa}_\eta)=\kappa_\eta=\sup(X)$.
Note that $\bar{N}^*_\infty|\widetilde{\kappa}_\eta=\widetilde{M}_X|\widetilde{\kappa}_\eta$ and $\tau^*\rest\widetilde{\kappa}_\eta=\tau\rest\widetilde{\kappa}_\eta$.
Since $\tau^*$ is elementary,
$\bar{N}^*_\infty$ is a premouse
extending $\widetilde{M}_X||\widetilde{\kappa}_\eta$,
with largest cardinal $\widetilde{\kappa}_\eta$, and $\bar{N}^*_\infty\sats$ ``$\widetilde{\kappa}_\eta$ is regular''.
But $\widetilde{M}_X\sats$ ``$\widetilde{\kappa}_\eta$ is singular''.

Now we claim that $\OR^{\bar{N}^*_\infty}<\widetilde{\kappa}_\eta^{+\widetilde{M}_X}$ and either:
\begin{enumerate}[label=--]
    \item $M|\kappa_\eta$ is passive (so $\widetilde{M}_X|\widetilde{\kappa}_\eta$ is passive) and $\bar{N}^*_\infty=\widetilde{M}_X||\OR^{\bar{N}^*_\infty}$, or
    \item $M|\kappa_\eta$ is active with extender $F$, $\widetilde{M}_X|\kappa_\eta$ is active with extender $\widetilde{F}$, and $\bar{N}^*_\infty=\Ult(\widetilde{M}_X,\widetilde{F})||\OR^{\bar{N}^*_\infty}$.
\end{enumerate}
Letting $\widetilde{\beta}=\OR^{\bar{N}^*_\infty}$, it easily follows
that $\tau(\beta)=\beta_\eta\in M_X$,
as desired.

Well, since $\widetilde{M}_X\sats$ ``$\widetilde{\kappa}_\eta$ is singular'',
we can't have $\widetilde{M}_X|\widetilde{\kappa}_\eta^{+\widetilde{M}_X}\sub \bar{N}^*_\infty$. Up until this stage of the argument we have just replicated Jensen's argument assuming
 $V=L$ (adjusted for the ``active'' case of condensation) and at this point, if $V=L$, the claim immediately follows now. But in general, here we need a little extra argument.
 If the claim fails,
there are $\chi,J,K$ such that $\chi<\min(\OR^{\bar{N}^*_\infty},\widetilde{\kappa}_\eta^{+\widetilde{M}_X})$,
 $\bar{N}^*_\infty||\chi=\widetilde{M}_X||\chi$,
 $J\pins\bar{N}^*_\infty$,
 $\rho_\om^J=\widetilde{\kappa}_\eta<\widetilde{\kappa}_\eta^{+J}=\chi$,
 $K\pins\widetilde{M}_X$,
 $\rho_\om^K=\widetilde{\kappa}_\eta<\widetilde{\kappa}_\eta^{+K}=\chi=\widetilde{\kappa}_\eta^{+J}$,
 and $J\neq K$. Note $J,K$ are sound.
 So $\mathscr{B}=(\widetilde{\kappa}_\eta,J,K)$
 is a sound bicephalus (see \cite{premouse_inheriting}).
 And $\mathscr{B}$ is iterable,
 because we have $\tau^*\rest J:J\to\tau^*(J)$
 and $\tau\rest K:K\to\tau(K)$,
 and $\tau,\tau^*$ agree below $\widetilde{\kappa}_\eta$,
 and by \cite{premouse_inheriting}. But then by \cite{premouse_inheriting},
 $J=K$, a contradiction.
\end{proof}
\begin{proof}[Proof of Theorem \ref{tm:diamond^+_kappa,kappa}]  We follow the proof of \cite[Theorem 10]{jensen_diamond}, adjusted as in the proof of Theorem \ref{tm:diamond^+}. (But actually, the adjustments are just as there, so we won't discuss them explicitly. So the details that follow are really just an exposition of Jensen's argument from \cite{jensen_diamond}, included here for convenience.)
      
      So work in $M$. Fix the $M$-least counterexample
     $\vec{S}=\left<S_\alpha\right>_{\alpha<\kappa}$ the ineffability of $\kappa$. Note that $\vec{S}$ is definable over $M|\kappa^{+M}$. We define $F:P_\kappa(\kappa)\to V$ just as before, except that we define $M_X=\Hull^{M|\kappa}(X\cup\{X,S_{\sup(X)}\})$ to define $F(X)$, when $X$ has limit ordertype).

     We claim $F$ works. Suppose not and let $A\sub\kappa$ be the least counterexample. We need $B\sub\kappa$ which works for $A$. So $A$ is definable over $M|\kappa^{+M}$ without parameters.

     Define $N_\alpha,\kappa_\alpha,\beta_\alpha$, etc, for for $\alpha<\kappa$, as before (but with $\gamma$ replaced by $\emptyset$). Define $B$ as before.
     Fix $X\in\pow_\kappa(\kappa)$ of limit ordertype,
     with $X\cap B$ cofinal in $\sup(X)$.
     Again, $\sup(X)=\kappa_\eta$ for some limit $\eta$.
     It is enough to see $\beta_\eta\in M_X$.
     This is like before, except that it might be that $\kappa_\eta$ is regular (hence regular in $M_X$), which needs to be handled. In this case, we will first observe that $S_{\kappa_\eta}\notin N_\eta$ (although $S_{\kappa_\eta}\in M_X$ by defintion).
     This is a direct consequence of the construction.
     That is, suppose $S_{\kappa_\eta}\in N_\eta$.
     Note that $\pi_\eta(\vec{S}\upharpoonright\kappa_\eta)=\vec{S}$.
     Let $T=\{\alpha<\kappa_\eta\bigm|S_\alpha=S_{\kappa_\eta}\cap\alpha\}$, so then $T\in N_\eta$.
     Now $N_\eta\sats$ ``$T$ is stationary'',
     because otherwise we can fix a club $C\sub\kappa_\eta$
     with $C\in N_\eta$ such that $S_\alpha\neq S_{\kappa_\eta}\cap\alpha$ for all $\alpha\in C$.
     So letting $C^*=\pi_\eta(C)$ and $S^*=\pi_\eta(S_{\kappa_\eta})$,
     $C^*\sub\kappa$ is club and $S_\alpha\neq S^*\cap\alpha$ for all $\alpha\in C^*$.
     But since $\kappa_\eta=\crit(\pi_\eta)$,
     $\kappa_\eta\in C^*$ and $S_{\kappa_\eta}=S^*\cap\kappa_\eta$, contradiction. So $N_\eta\sats$ ``$T$ is stationary''. But then clearly $\pi_\eta(T)\sub\kappa$ is stationary and $S_\alpha=\pi_\eta(S_{\kappa_\eta})\cap\alpha$ for all $\alpha\in T$,
     and hence $S_\alpha=S_\beta\cap\alpha$
     for all $\alpha,\beta\in T$ with $\alpha<\beta$, contradicting the choice of $\vec{S}$.

     The remainder of the argument is much as before.
     But with notation as there, one also needs a little more thought to rule out the possibility, in the case that $\kappa_\eta$ is regular (hence passive)
     that $\widetilde{\kappa}_\eta^{+\widetilde{M}_X}\leq\OR^{\bar{N}^*_\infty}$. This uses what we established in the previous paragraph. For suppose
     $\widetilde{\kappa}_\eta^{+\widetilde{M}_X}\leq\OR^{\bar{N}^*_\infty}$. Let $\tau(\widetilde{S}_{\kappa_\eta})=S_{\kappa_\eta}$. We get $\widetilde{S}_{\kappa_\eta}\in \bar{N}^*_\infty$.
     But then $\tau^*(\widetilde{S}_{\kappa_\eta})=S_{\kappa_\eta}\in N_\eta$;
     this is because $\tau(\widetilde{S}_{\kappa_\eta})=\kappa_\eta$ and $\tau^*\upharpoonright(\bar{N}^*_\infty|\widetilde{\kappa}_\eta)=\tau\upharpoonright(\widetilde{M}_X|\widetilde{\kappa}_\eta)$ and $\tau``\widetilde{\kappa}_\eta$ is cofinal in $\kappa_\eta$ and
      $\widetilde{S}_{\kappa_\eta}\cap\alpha\in \bar{N}^*_\infty|\widetilde{\kappa}_\eta$ for each $\alpha<\widetilde{\kappa}_\eta$. But we saw $S_{\kappa_\eta}\notin N_\eta$, contradiction. This completes the proof.
\end{proof}

\section{$\diamondsuit^*_{\kappa\om_1}$}

\begin{proof}[Proof of Theorem \ref{tm:diamond^*}]
Work in $M$. Given $X\in\pow_{\om_1}(\kappa)$, let $\bar{M}_X=\cHull_\om^{M|\kappa}(X)$
and $\pi_X:\bar{M}_X\to M|\kappa$ the uncollapse map.

Say that a premouse $J\in\HC$
is \emph{$X$-good} iff $\bar{M}_X\pins J$,
$\OR^{\bar{M}_X}$ is a $J$-cardinal,
$J$ is sound, $\rho_\om^J=\OR^{\bar{M}_X}$, and there are $J',\sigma$
such that $J'\pins M|\kappa^{+}$,
 $\rho_\om^{J'}=\kappa$
and $\sigma:J\to J'$
is elementary and  $\pi_X\sub\sigma$.

\begin{clm}Let $J,K$ be $X$-good.
Then $J\ins K$ or $K\ins J$.
\end{clm}
\begin{proof}
Suppose not. Then by minimizing on the ordinal heights of $J,K$,
we get $J,K$ such that $J\neq K$
but $J||\bar{\kappa}_X^{+J}=K||\bar{\kappa}_X^{+K}$.
So $\mathscr{B}=(\bar{\kappa}_X,J,K)$
is a non-trivial sound bicephalus
(see \cite{premouse_inheriting}).
But since we have $\sigma_J:J\to J'$
and $\sigma_K:K\to K'$
and $\sigma_J\rest\bar{\kappa}_X=\pi\rest\bar{\kappa}_X=\sigma_K\rest\bar{\kappa}_X$,
$\mathscr{B}$ is iterable (see \cite{premouse_inheriting}).
So by \cite{premouse_inheriting}, $J=K$, a contradiction.
\end{proof}

Let $N_X$ be the premouse which is the stack of all $X$-good premice.
So note that $\OR^{N_X}\leq\om_1$.

\begin{clm}
$\OR^{N_X}<\om_1$.
\end{clm}
\begin{proof}
Suppose otherwise. Then
there is no largest $X$-good $J$,
so $N_X\sats$ ``$\bar{\kappa}_X$ is the largest cardinal'', and $\om_1^{N_X}<\om_1=\OR^{\bar{N}_X}$.

For each $X$-good $J$,
let $J'$ be the least segment of $M|\kappa^+$ witnessing the fact that $J$ is $X$-good.
Note that there is a unique elementary $\sigma:J\to J'$ such that $\pi_X\sub\sigma$; denote it $\sigma_J$.

Now we can treat $N_X$ as a ``bicephalus'' $\mathscr{B}$ with infinitely many models (the good $J$'s). The general calculations go through with 
$\mathscr{B}$ just like for standard (2-model) bicephali.

Further, $\mathscr{B}$ is $(0,\om_1+1)$-iterable in $V$, since
 all $\sigma_J$ agree with $\pi$, and the usual iterability proof adapts immediately to this context. But then $N_X$ is $(0,\om_1+1)$-iterable, since iterating $N_X$ is equivalent to iterating $\mathscr{B}$.
And if hypothesis 
\ref{item:self-it} holds, then similarly, $\mathscr{B}$ and  $N_X$ are  $(0,\omega_1+1)$-iterable in $M$, because working in $M$, letting $\eta<\kappa^{+M}$ be above $\OR^{J'}$ for each good $J$,
we can fix a size $\aleph_1$ substructure $X\preccurlyeq N|\eta$ containing the range of each $\sigma_J$, and then letting $\bar{N}$ be the transitive collapse of $X$, use the iterability of $\bar{N}$ to deduce that of $\mathscr{B}$.

So in either case, both $M|\om_1^M$ and $N_X$ are $(0,\om^M_1+1)$-iterable in some universe in which $\om_1^M=\om_1$ (either $V$ or $M$). Work in that universe. We can successfully compare them, with a comparison of length $\leq\om_1$. But since $\om_1^{N_X}<\om_1$ and $N_X$ has a largest cardinal, this is impossible.
\end{proof}

Now define $F(X)=\{X\cap\pi_X``x\bigm|x\in N_X\}$.

We claim $F$ works. For let $A\sub\kappa$. Let $J'\pins M|\kappa^+$ be such that $\rho_\om^{J'}=\kappa$
and $A\in J'$.
Note that for a club of $X\in\pow_{\om_1}(\kappa)$,
there is an $X$-good $J\pins N_X$
and a witnessing embedding $\sigma:J\to J'$
such that 
$\sigma(\rho_\om^J)=\sigma(\OR^{\bar{M}_X})=\kappa=\rho_\om^{J'}$
and $A\in\rg(\sigma)$.
But then letting $\sigma(x)=A$,
we get $X\cap\pi_X``x=A\cap X\in F(X)$,
as desired.
\end{proof}
\bibliographystyle{plain} 
\bibliography{bibliography}

\begin{thebibliography}{1}

\bibitem{jensen_diamond}
Ronald Jensen.
\newblock Some combinatorial properties of {$L$} and {$V$}.
\newblock Handwritten notes available at
  \url{https://www.math.uni-bonn.de/~raesch/jensen/}, 1969.

\bibitem{sile}
Ralf Schindler and John Steel.
\newblock The self-iterability of {$L[E]$}.
\newblock {\em Journal of Symbolic Logic}, 74(3):751--779, 2009.

\bibitem{premouse_inheriting}
Farmer Schlutzenberg.
\newblock A premouse inheriting strong cardinals from {$V$}.
\newblock {\em Annals of Pure and Applied Logic}, 171(9), 2020.

\bibitem{odle}
Farmer Schlutzenberg.
\newblock Ordinal definability in {$L[\mathbb {E}]$}.
\newblock {\em The Journal of Symbolic Logic}, 91(2):489--537, 2026.

\end{thebibliography}
\end{document}